\documentclass[12pt]{amsart}
\usepackage{amscd,amsmath,amsthm,amssymb}
\usepackage{amsfonts,amssymb,amscd,amsmath,enumerate,verbatim}
\usepackage[left]{lineno}
\def\NZQ{\mathbb}               % the font for N,Z,Q,R,C

\def\ZZ{{\NZQ Z}}
\def\RR{{\NZQ R}}

\def\G{{\mathcal G}}
\def\F{{\mathcal F}}

\def\P{{\mathcal P}}
\def\Mc{{\mathcal M}}

\def\Ac{{\mathcal A}}

\def\Pc{{\mathcal P}}
\def\Pc{{\mathcal P}}
\def\Pc{{\mathcal P}}
\def\Fc{{\mathcal F}}
\def\xb{{\mathbf x}}

\def\eb{{\mathbf e}}
\def\wb{{\mathbf w}}
\def\vb{{\mathbf v}}
\def\ub{{\mathbf u}}
\def\opn#1#2{\def#1{\operatorname{#2}}} % to make operators
\opn\chara{char} \opn\length{\ell} \opn\pd{pd} \opn\rk{rk}
\opn\projdim{proj\,dim} \opn\injdim{inj\,dim} \opn\rank{rank}
\opn\depth{depth} \opn\grade{grade} \opn\height{height}
\opn\embdim{emb\,dim} \opn\codim{codim}
\opn\Cl{Cl}

\opn\Tr{Tr} \opn\bigrank{big\,rank}
\opn\superheight{superheight}\opn\lcm{lcm}
\opn\trdeg{tr\,deg}%\emph{
	\opn\reg{reg} \opn\lreg{lreg} \opn\ini{in} \opn\lpd{lpd}
	\opn\size{size} \opn\sdepth{sdepth}
	\opn\link{link}\opn\fdepth{fdepth}\opn\lex{lex}
	\opn\tr{tr}
	\opn\type{type}
	\opn\gap{gap}
	\opn\arithdeg{arith-deg}
	\opn\revlex{revlex}
	\opn\div{div} \opn\Div{Div} \opn\cl{cl} \opn\Cl{Cl}
	\opn\Spec{Spec} \opn\Supp{Supp} \opn\supp{supp} \opn\Sing{Sing}
	\opn\Ass{Ass} \opn\Min{Min}\opn\Mon{Mon}
	\opn\Ann{Ann} \opn\Rad{Rad} \opn\Soc{Soc}
	\opn\Im{Im} \opn\Ker{Ker} \opn\Coker{Coker} \opn\Am{Am}
	\opn\Hom{Hom} \opn\Tor{Tor} \opn\Ext{Ext} \opn\End{End}
	\opn\Aut{Aut} \opn\id{id}
	\def\F{{\mathcal F}}
	\opn\nat{nat}
	\opn\pff{pf}%   \pf exists already
	\opn\Pf{Pf} \opn\GL{GL} \opn\SL{SL} \opn\mod{mod} \opn\ord{ord}
	\opn\Gin{Gin} \opn\Hilb{Hilb}\opn\sort{sort}
	\opn\PF{PF}\opn\Ap{Ap}
	\opn\mult{mult}
	\opn\bight{bight}
	\opn\div{div}
	\opn\Div{Div}
	\opn\aff{aff}
	\opn\relint{relint} \opn\st{st}
	\opn\lk{lk} \opn\cn{cn} \opn\core{core} \opn\vol{vol}  \opn\inp{inp} \opn\nilpot{nilpot}
	\opn\link{link} \opn\star{star}\opn\lex{lex}\opn\set{set}
	\opn\width{wd}
	\opn\Fr{F}
	\opn\QF{QF}
	\opn\G{G}
	\opn\type{type}\opn\res{res}
	\opn\conv{conv}
	\opn\Deg{Deg}
	\opn\Sym{Sym}
	\opn\Con{Con}
	\opn\gr{gr}
	
	\def\pot#1#2{#1[\kern-0.28ex[#2]\kern-0.28ex]}

	\opn\dirlim{\underrightarrow{\lim}}
	\opn\inivlim{\underleftarrow{\lim}}
	\let\union=\cup
	\let\sect=\cap
	\let\dirsum=\oplus
	
	\let\iso=\cong

	\let\to=\rightarrow
	
	\def\Implies{\ifmmode\Longrightarrow \else
		\unskip${}\Longrightarrow{}$\ignorespaces\fi}
	\def\implies{\ifmmode\Rightarrow \else
		\unskip${}\Rightarrow{}$\ignorespaces\fi}
	\def\iff{\ifmmode\Longleftrightarrow \else
		\unskip${}\Longleftrightarrow{}$\ignorespaces\fi}

	\let\:=\colon
	\newtheorem{Theorem}{Theorem}[section]
	\newtheorem{Lemma}[Theorem]{Lemma}
	\newtheorem{Corollary}[Theorem]{Corollary}
	\newtheorem{Proposition}[Theorem]{Proposition}

	\newtheorem{Example}[Theorem]{Example}

	\let\epsilon\varepsilon
	\let\kappa=\varkappa
	\def\qed{\ifhmode\textqed\fi
		\ifmmode\ifinner\quad\qedsymbol\else\dispqed\fi\fi}
	\def\textqed{\unskip\nobreak\penalty50
		\hskip2em\hbox{}\nobreak\hfil\qedsymbol
		\parfillskip=0pt \finalhyphendemerits=0}
	\def\dispqed{\rlap{\qquad\qedsymbol}}
	
	\opn\dis{dis}
	\def\pnt{{\raise0.5mm\hbox{\large\bf.}}}
	
	\opn\Lex{Lex}
	
\begin{document}
		
		\title{The divisor class group of a discrete polymatroid }
		
		\author {J\"urgen Herzog, Takayuki Hibi, Somayeh Moradi and Ayesha Asloob Qureshi}

		\address{J\"urgen Herzog, Fachbereich Mathematik, Universit\"at Duisburg-Essen, Campus Essen, 45117
			Essen, Germany} \email{juergen.herzog@uni-essen.de}

		\address{Takayuki Hibi, Department of Pure and Applied Mathematics,
	Graduate School of Information Science and Technology, Osaka
	University, Suita, Osaka 565-0871, Japan}
\email{hibi@math.sci.osaka-u.ac.jp}

		\address{Somayeh Moradi, Department of Mathematics, Faculty of Science, Ilam University,
			P.O.Box 69315-516, Ilam, Iran}
		\email{so.moradi@ilam.ac.ir}
		
			\address{Ayesha Asloob Qureshi, Sabanci University, Faculty of Engineering and Natural Sciences, Orta Mahalle, Tuzla 34956, Istanbul, Turkey}
	
		\email{aqureshi@sabanciuniv.edu}
		
		\dedicatory{ }
	\keywords{toric rings, discrete polymatroids, class group, canonical module}
	\subjclass[2010]{Primary 13A02; 13P10, Secondary 05E40}
	%\thanks{The second author was supported by JSPS KAKENHI 19H00637.}
	\thanks{Takayuki Hibi is partially supported by JSPS KAKENHI 19H00637. Somayeh Moradi is supported by the Alexander von Humboldt Foundation. Ayesha Asloob Qureshi is supported by The Scientific and Technological Research Council of Turkey - T\"UBITAK (Grant No: 122F128).}
		
\begin{abstract}
In this paper we introduce  toric rings of multicomplexes.  We show how to compute the  divisor class group and the class of the canonical module when the toric ring is normal.  In the  special case that the multicomplex is a discrete polymatroid,  its toric  ring is studied deeply  for several classes of polymatroids. 
\end{abstract}

		\maketitle
		
		\setcounter{tocdepth}{1}
		%\tableofcontent		
\section*{Introduction}	

In the previous paper \cite{HHMQ} the authors introduced toric rings of simplicial complexes. In this paper we extend this concept to multicomplexes with a special focus on polymatroids. 
	
In Section 1 we recall the relevant terminology and background. 
In Section 2 we present the general  framework of toric rings of multicomplexes and show how to compute the divisor class group and the class of the canonical module when the toric ring is normal.   The methods for the proofs are similar to those used in \cite {HHMQ}.  It turns out that their divisor class group  is an abelian group with just one relation.

In Section 3 the results of Section 2 are made more explicit when the multicomplex is a discrete polymatroid. There has been many research works on discrete polymatroids and their base rings. We refer the reader to \cite{HH, HHV, Lu, Sch, V} and the references therein.  In this paper we determine the divisor class group and the canonical class of toric rings of discrete polymatroids, and as an application  we recover the Gorenstein criterion discussed in \cite[Example 7.4(b)]{HH}.

Section 4 deals with transversal polymatroids. As one of the main results it is shown  in Corollary~\ref{n_and_d} that for any integer $r\geq 1$ and any integer $d\geq 0$, there exists a transversal polymatroid for which the  divisor class group of its toric ring is isomorphic to $\ZZ^{r-1}\dirsum \ZZ/d\ZZ$. In Theorem~\ref{finite} those  transversal polymatroids are characterized for which the divisor class group of its toric ring is a finite cyclic group, and in Theorem~\ref{2_prime_THEOREM} those for which the divisor class group  is isomorphic to $\ZZ\dirsum \ZZ/d\ZZ$. Other families of transversal polymatroids are also considered. 

In the last section we classify all discrete polymatroids of Veronese type whose toric ring is Gorenstein. For the  base ring of  discrete polymatroids of Veronese  type this classification was achieved in \cite{DH}.

\section{Preliminaries and background}\label{one}

In this section we recall the basic concepts which are relevant for this paper.  We denote the set of non-negative integers by $\ZZ_+$ and the set of non-negative real numbers by $\RR_+$.  For two vectors $\ub=(u_1,\ldots,u_n)$ and $\vb=(v_1,\ldots,v_n)$ in $\RR_+^n$ we write $\ub\leq \vb$ if $u_i\leq v_i$ for all $i$. Moreover, write $\ub<\vb$ if  $\ub\leq \vb$  and  $\ub\neq \vb$.  Let $\eb_1,\ldots,\eb_n$ be the standard basis of $\ZZ^n$.  A {\em multicomplex} on the ground set $[n]$ is a nonempty  finite set  $\mathcal{M}\subset \ZZ_+^n$ such that 

\begin{enumerate}
	\item  for any $\vb\in \mathcal{M}$ and $\ub\in \ZZ_+^n$ with $\ub\leq \vb$, one has $\ub\in \mathcal{M}$. 
	\item   $\eb_i\in \mathcal{M}$ for any $1\leq i\leq n$.   
\end{enumerate} 

Note that a simplicial complex on $[n]$ is in fact a  multicomplex consisting of $0,1$-vectors.

Let $\mathcal{M}$ be a multicomplex on $[n]$. A vector $\vb\in \mathcal{M}$ is called a {\em facet} of $\mathcal{M}$ if there exists no $\wb\in \mathcal{M}$ with $\vb<\wb$. The set of  all facets of $\mathcal{M}$ is denoted by $\mathcal{F}(\mathcal{M})$. If $\Fc(\mathcal{M})=\{\vb_1,\ldots,\vb_m\}$, then we write $\mathcal{M}=\langle \vb_1,\ldots,\vb_m\rangle$. Let  $K$ be a field. For a vector $\vb=(v_1,\ldots, v_n)\in \mathcal{M}$, we define the monomial  $\xb^\vb=\prod_{i=1}^n x_i^{v_i}$ in the polynomial ring $K[x_1,\ldots,x_n]$. The {\em toric ring of $\mathcal{M}$} is defined to be  the subalgebra 
\[
R_{\mathcal{M}}=K[\xb^\vb t\: \vb\in \mathcal{M}],
\]
of the polynomial ring $S=K[x_1,\ldots,x_n,t]$. The algebra $R_\mathcal{M}$	
has a $K$-basis consisting  of monomials of $S$. If $f=x_1^{a_1}\cdots x_n^{a_n}t^k$ belongs to $R_\mathcal{M}$, we set $\deg f=k$. By this grading  $R_\mathcal{M}$ is a standard graded $K$-algebra.

Any monomial $\xb^{\vb}t^b\in K[x_1,\ldots,x_n,t]$ can be identified  with its exponent vector $(\vb,b)\in \ZZ^{n+1}$. Then the monomial $K$-basis of $R_\mathcal{M}$ corresponds to an affine semigroup $S\subset \ZZ^{n+1}$ which is generated by the  lattice points  $p_\vb=\sum_{i=1}^n v_i\eb_i+\eb_{n+1}$ in $\ZZ^{n+1}$, where $\vb\in \mathcal{M}$. 

Let $\ZZ S$ be the smallest subgroup of $\ZZ^{n+1}$  containing $S$ and let $\RR_+ S\subset \RR^{n+1}$ be the smallest cone containing $S$.  In our case, $\ZZ S=\ZZ^{n+1}$. Since we assume $R$ is normal,  Gordon's lemma \cite[Proposition 6.1.2]{BH} guarantees that $S=\ZZ^{n+1}\sect \RR_+ S$. 

%A  hyperplane $H$, defined as the set of solutions  of   the linear equation $f(\xb):=a_1x_1+\cdots +a_{n+1}x_{n+1}=0$, is a supporting hyperplane of the cone $\RR_+ S$,  if $H\sect  \RR_+ S\neq \emptyset$ and $f(\xb)\geq 0$ for all $\cb\in \RR_+ S$. A subset $\F$  of  $\RR_+ S$ is called a face of  $\RR_+ S$, if there exists  a supporting hyperplane $H$ of $\RR_+ S$  such that $\F=H\sect  \RR_+ S$. 

By \cite[Corollary 4.35]{BG} all minimal prime ideals of a monomial ideal in $R_\mathcal{M}$ are monomial prime ideals. In particular, they are generated by subsets of the generators $\xb^\vb t$ of $R_\mathcal{M}$. Moreover,  it follows from  \cite[Proposition 2.36 and Proposition 4.33]{BG} that  $P\subset R_\mathcal{M}$ is a monomial prime ideal of $R_\mathcal{M}$  if and only if there exists a face  $\F$ of $\RR_+ S$ such that $P =(x^\vb t\: p_\vb\not\in \F)$. In other words, $P$ is a monomial prime ideal if and only if there exists a supporting hyperplane $H$ of $ \RR_+ S$ such  that
\[
P=(x^\vb t\: \vb\in \mathcal{M} \text{ and } f(p_\vb)>0),
\]
where $f$ is a linear form defining $H$.

The supporting hyperplane $H$ of a facet is uniquely determined. Since $H$ is spanned by lattice points, a linear form $f=\sum_{i=1}^{n+1}c_ix_i$ defining $H$ has rational coefficients. By clearing denominators we may assume that all $c_i$ are integers, and then dividing $f$ by the greatest common divisor of the $c_i$, we may furthermore assume that $\gcd(c_1,\ldots,c_{n+1})=1$. Then this normalized linear form $f$ is uniquely determined by $H$. It has the property that $f(H)=0$ and $f(\ZZ^{n+1})\sect \ZZ_+=\ZZ_+$. Indeed, since $\gcd(c_1,\ldots,c_{n+1})=1$, there exist $p= (b_1,\ldots,b_{n+1})\in \ZZ^{n+1}$ with $\sum_{i=1}^{n+1}b_ic_i=1$, which implies that $f(p)=1$.

If $P$ is  a height $1$ monomial prime ideal, then  $P=(x^\vb t\: p_\vb\not\in \F)$,  where $\F$ is a facet of $\RR_+S$. Let $H$ be the supporting hyperplane of $\F$. Then we  call the normalized linear form which defines $H$, the {\em support form} associated  to $P$.

For  a vector $\vb=(v_1,\ldots,v_n)\in \ZZ^n$ we set $|\vb|=\sum_{i=1}^nv_i$,  and for a subset $A\subset [n]$ we set $\vb(A)=\sum_{i\in A}v_i$.  In particular, $\vb(\emptyset)=0$ and $\vb([n])=|\vb|$. 

A {\em discrete polymatroid} on the ground set $[n]$ is a  multicomplex $\P$ satisfying the following property: if $\vb=(v_1,\ldots,v_n)$ and $w=(w_1,\ldots,w_n)$ belong to $\P$ and $|\vb|<|\wb|$, then there exists $i\in [n]$ with $v_i<w_i$ such that $\vb+\eb_i\in\P$. 

The {\em ground set rank function} $\rho\: 2 ^{[n]}\to \ZZ_+$ of a discrete polymatroid $\P$ is defined to be
\[
\rho(A)=\max\{ \vb(A)\: \vb\in \P\}.
\]
The ground set rank function $\rho$ of $\P$ has the following properties:
\begin{enumerate}
	\item[(i)] $\rho(A)\leq \rho(B)$, if $A\subseteq B\subseteq [n]$, and
	\item[(ii)] $\rho(A)+\rho(B)\geq \rho(A\union B)+\rho(A\sect B)$ for all $A, B\subset [n]$.
\end{enumerate}

A nonempty set $A\subseteq [n]$ is called {\em $\rho$-closed}, if $\rho(A)<\rho(B)$ for all $B\subseteq [n]$ properly containing $A$, and $A$ is called {\em separable}, if there exist nonempty subset $A_1$ and $A_2$ of $A$ with $A=A_1\union A_2$ and $A_1\sect A_2=\emptyset$   such that $\rho(A)=\rho(A_1)+\rho(A_2)$.  The set $A$ is  called {\em inseparable} if it is not separable.

\section{The toric  face ring of a multicomplex}\label{two}
 In this section we study the class group of normal toric rings of multicomplexes. 
 The following theorem generalizes \cite[Theorem 1.1]{HHMQ}
with exactly the same argument. %Moreover, we investigate the  height one monomial prime ideals in an arbitrary ring $R_{\mathcal{M}}$  which are required to understand the class group and the class of the canonical module of  normal toric rings of multicomplexes. 

\begin{Theorem}
	\label{classgroup}
	Let $\mathcal{M}$ be a multicomplex on $[n]$ such that $R=R_\mathcal{M}$ is normal, and let $P_1,\ldots,P_r$ be the minimal prime ideals of $(t)$, where $t\in R$ is the element corresponding to the zero vector in $\mathcal{M}$. Then $\Cl(R)$ is generated by the classes $[P_i]$, $i=1,\ldots,r$. Since $R_{P_i}$ is a discrete valuation ring, we have $tR_{P_i}=P_i^{a_i}R_{P_i}$ with $a_i\in \ZZ$ for $i=1,\ldots,r$. Then  $\sum_{i=1}^ra_i[P_i]=0$ is the only generating relation among these generators of $\Cl(R)$. %In particular, $\Cl(R_\Delta)$ is free of rank $r-1$,  if $(t)$ is a radical ideal. 
\end{Theorem}

The following lemma and proposition are needed for studying the canonical class of $R_\mathcal{M}$. 

\begin{Lemma}(see \cite[Lemma 1.6]{HHMQ})
	\label{vr}
	Let $\mathcal{M}$ be a multicomplex on $[n]$ such that $R=R_\mathcal{M}$ is normal, and  let $P$ be a monomial  prime ideal of $R$ of height one. Furthermore,  let $f$ be the support form  associated with $P$, and let $\vb_f=(c_1,\ldots,c_{n+1})$ be the coefficient vector of $f$. Then  the following holds:
	
	If $u\in Q(R)$ is a monomial with exponent vector $\vb_u$, then $uR_P=P^aR_P$, where 
	$
	a=\langle \vb_f,\vb_u\rangle.
	$
	Here $\langle -,-\rangle$ denotes the standard inner product in  $\RR^{n+1}$. 
\end{Lemma}

\begin{Proposition}\label{othermin}
	Let $\mathcal{M}$ be a  multicomplex on the ground set $[n]$. For $i=1, \ldots, n$, let $Q_i=(\xb^\vb t \: \vb\in\mathcal{M},\ v_i>0)$.  Then $\{Q_1, \ldots, Q_n\}$ is the set of height one monomial prime ideals of $R_\mathcal{M}$ which do not contain $t$.
\end{Proposition}

\begin{proof}
For $i=1, \ldots, n$, let $f_i(\xb)=x_i$ and $H_i=\{\xb \: f_i(\xb)=0\}$. We claim that $H_i$ is a supporting hyperplane of a facet of $\RR_+S$.  For all $\vb \in \mathcal{M}$, we have $f_i(p_\vb)=v_i \geq 0$. Hence $H_i$ is a supporting hyperplane, and the points $\eb_1, \ldots, \widehat{\eb_i}, \ldots, \eb_{n+1}$ lie on the hyperplane. This shows that $H_i$ is a supporting hyperplane of a facet of $\RR_+S$. Hence $P=(\xb^\vb t \: \vb\in\mathcal{M},\ f_i(p_\vb) >0)$ is a height one monomial prime ideal of $R_\mathcal{M}$ and $P=Q_i$.
	
Now, let $P$ be a height one monomial prime ideal of $R_\mathcal{M}$ with $t \notin P$. First we claim that $x_it \in P$, for some $i$. %If $x_i^rt \in P$, for some $i$ and $r$, then $(x_it)^r=(x_i^rt)t^{r-1}\in P$. Hence $x_it \in P$, as desired. 
%Now, suppose that $x_i^rt \notin P$, for all $i$ and $r$. 
Suppose that this is not the case. Then there exists a nonzero $\vb\in \mathcal{M}$ such that $\xb^\vb t\in P$ and for any $\ub<\vb$, $\xb^\ub t\notin P$. 
Let $j$ be an integer with $v_j>0$ and let $\wb=\vb-\eb_j$. Then 
  $(x^\wb t)(x_jt)=(\xb^\vb t)(t) \in P$, while  $x^\wb t \notin P$ and $x_jt\notin P$, a contradiction. So the claim is proved and there exists $i$ such that $x_it \in P$.    
 Let $\vb\in\mathcal{M}$ with $v_i>0$.     Then $(\xb^\vb t)(t)=(x_it)(\xb^{\vb-e_i}t) \in P$. Since $t \notin P$, we have $\xb^\vb t\in P$. This shows that $Q_i \subseteq P$. Since $P$ and $Q_i$ both have height one, we obtain $P=Q_i$. 
\end{proof} 

Let $\omega_{R_{\mathcal{M}}}$ be the canonical module of $R_{\mathcal{M}}$. 
By \cite[Corollary~3.3.19]{BH}, $\omega_{R_{\Mc}}$ is a divisorial ideal and corresponds to the relative interior of the cone $\RR_+S$, see \cite[Theorem~6.3.5(b)]{BH}.
Let $P_1.\ldots,P_r$ be the  height one monomial prime ideals of $R_{\mathcal{M}}$ which contain $t$, and for each $j$ let $f_j$ be the support form  associated with $P_j$. Let $\vb_{f_j}=(c_{1,j},\ldots,c_{n+1,j})$ be the coefficient vector of $f_j$. Then by Proposition~\ref{othermin}, $P_1, \ldots, P_r, Q_1, \ldots, Q_n$ are the sets of all height one monomial prime ideals of $R_{\mathcal{M}}$. 
By \cite[Theorem~6.3.5(b)]{BH} we have  
\begin{equation}\label{eq:intersection}
	\omega_{R_{\Mc}}=( \bigcap_{i=1}^rP_i) \cap (\bigcap_{j=1}^n Q_j).
\end{equation} 

\begin{Theorem}\label{verygood}
	Let	$\mathcal{M}$ be a multicomplex on the ground set $[n]$ such that $R_{\mathcal{M}}$ is normal. Then with the notation introduced above, we have 
	\[
	[\omega_{R_{\Mc}}]= \sum_{j=1}^r (1-c_{1,j}-\cdots-c_{n,j}) [P_j].
	\]
\end{Theorem}

\begin{proof}
	To simplify the notation, we set $R=R_\mathcal{M}$. It follows from (\ref{eq:intersection}) that $[\omega_{R}]= \sum_{j=1}^r  [P_j]+\sum_{k=1}^n  [Q_k]$. %By our assumption on $(t)$ it follows that $\sum_{j=1}^r  [P_j]=0$. 
	%	Therefore,  $[\omega_{R}]=\sum_{k=1}^n  [Q_k]$. 
For a fixed integer $1\leq i\leq n$, let 
$[x_i]= \sum_{j=1}^r a_j [P_j]+\sum_{k=1}^n b_k [Q_k]$. By \cite[Lemma~1.6]{HHMQ} we have $
a_j =\langle \vb_{f_j},\eb_i\rangle=c_{i,j}$, where $\eb_i\in \ZZ^{n+1}$ is a standard basis element. 
For each $1\leq j\leq n$, the associated support form of $Q_j$ is $f'_j(x)=x_j$ (see Proposition~\ref{othermin}). Hence $\vb_{f'_j}=e_j$ and
\[
b_j =\langle \eb_j,\eb_i\rangle= \left\{
\begin{array}{ll}
	1,  &   \text{if $j=i$,}\\
	0, & \text{otherwise.}\\
\end{array}
\right. \]
Hence
	\[
	[x_i] = (\sum_{j=1}^r c_{i,j}[P_j])+[Q_i].
	\]
from which we conclude that $[Q_i]=-\sum_{j=1}^r c_{i,j}[P_j]$. This implies the desired equality.
\end{proof}

Remark. Theorem~\ref{classgroup} holds true when we replace $\mathcal{M}$ by any finite set of vectors in $\ZZ_+^n$ which includes $e_1,\ldots,e_n$ and the zero vector.

\section{The class group and the canonical class of the toric ring of a discrete polymatroid}

Discrete polymatroids are special, but important classes of multicomplexes. In this section we apply the results of Section~\ref{two} to obtain the class group and the canonical class of discrete polymatroids.

Discrete polymatroids are particularly nice multicomplexes. Indeed, by \cite{E}  one has  

\begin{Theorem}[Edmonds]
\label{normalnice}
Let  $\P$ be a discrete polymatroid.  Then $R_\P$ is normal. 
\end{Theorem}

\medskip
Let $\P$ be a discrete polymatroid, and let $S\subset \ZZ^{n+1}$ be the affine semigroup which is generated by the lattice points corresponding to the generators $x^\vb t$  of $R_\P$. We need to determine the hyperplanes defining  the facets of the cone $\RR_+S$. 

For each $A\subseteq [n]$ which is  $\rho$-closed and $\rho$-inseparable, we consider the hyperplane $H_A$ defined by the linear form 
\begin{eqnarray}\label{edmondsform}
f_A(x)=-\sum_{i\in A}x_i+\rho(A)x_{n+1},
\end{eqnarray}
and for $i=1,\ldots,n$, let  $H_i$ be the hyperplane defined by the linear form $f_i(x)=x_i$.

The following result is crucial for our considerations 

\begin{Theorem}[\cite{E}]
\label{crucial}
The hyperplanes $H_A$ and the hyperplanes $H_i$ introduced above are the supporting hyperplanes of the facets of the cone $\RR_+S$ attached to the polymatroid $\P$. 
\end{Theorem}

We denote by $P_A$ the monomial prime ideals of $R_\P$ determined by the hyperplanes $H_A$ and by $Q_i$ the monomial prime ideals determined by the hyperplanes $H_i$.  

Now we may apply  the results of  Section~\ref{two} and obtain

\begin{Theorem}
\label{basic}
Let $\P$ be a polymatroid on the ground  set $[n]$, and let $\rho$ be its ground set rank function. Let $\Ac$ be the set of  $\rho$-closed and $\rho$-inseparable subsets  of $[n]$.  Then $\Cl(R_\P)$ is generated by the classes $[P_A]$ with $A\in \Ac$.   Moreover, 
$
\sum_{A\in \Ac}\rho(A)[P_A]=0
$
is the only generating relation among these  generators of $\Cl(R_\P)$. 

In particular, $\Cl(R_\P)\iso \ZZ^{r-1}\dirsum \ZZ/d\ZZ$, where 
$r=|\Ac|$  and $d=\gcd\{\rho(A)\:\; A\in \Ac\}$.
\end{Theorem}

\begin{proof}
The prime ideals $P_A$ are precisely the minimal prime ideals of $(t)$. Therefore,  the divisor classes $[P_A]$ with $A\in \Ac$ generate $\Cl(R_\P)$. By Edmond's theorem,  the coefficient vector of the support form of $P_A$ is $\vb_A=-\sum_{i\in A}e_i+\rho(A)e_{n+1}$. Therefore, by Theorem~\ref{classgroup} and Lemma~\ref{vr}, the generating relation of $\Cl(R_\P)$ is $\sum_{A\in \Ac} \rho(A)[P_A]=0$, as asserted. 

The resulting group structure of $\Cl(R_\P)$ is an immediate consequence of the statements before. 
\end{proof}

For the canonical class of $R=R_\P$ we have the following presentation.

\begin{Theorem}
\label{canonical class}
$
[\omega_R]=\sum_{A\in \Ac}(|A|+1)[P_A].
$
\end{Theorem}

\begin{proof}
Let $\vb_A=(c_{1,A}, \ldots, c_{n+1,A})$ be the coefficient vector of the linear form  $f_A$. It follows from (\ref{edmondsform}) that $c_{i,A}=-1$ if $i\in A$,  and $c_{i,A}=0$ if $i\not\in A$. Therefore, Theorem~\ref{verygood} yields the desired result.
\end{proof}

\begin{Corollary}
\label{gorenstein}
$R_\P$ is Gorenstein if and only if there exists an integer $a$ such that 
\[
|A|+1=a\rho(A)
\]
for all $\rho$-closed and $\rho$-inseparable  $A\subseteq [n]$. 
\end{Corollary}

\begin{proof}
The ring $R_\P$ is Gorenstein if and only if $\omega_{R_\Pc}$ is a principal ideal, which by Theorem~\ref{canonical class} is the case if and only if $\sum_A(|A|+1)[P_A]=0$. Thus the desired result  follows from Theorem~\ref{basic}.
\end{proof}

As a first example consider for any integer $d>0$  the discrete polymatroid $\P_d=\{v\:\; |v|\leq d\}$ on the ground set $[n]$. For $\P_d$ the only $\rho$-closed and $\rho$-inseparable set is $[n]$ with $\rho([n])=d$. Thus,  Theorem~\ref{basic} implies that $\Cl(R_{P_d})\iso \ZZ/d\ZZ$, and from Corollary~\ref{gorenstein} it follows that $R_{\P_d}$ is Gorenstein if and only if $d$ divides $n+1$. 

\medskip
Next we present another application of Theorem~\ref{basic} and Corollary~\ref{gorenstein}. 

\begin{Example}
{\em Let $\vb=(v_1,\ldots,v_n)\in \ZZ^n$ be a vector with $v_i\neq 0$ for all $i$, and let $$\P=\{\wb\in \ZZ^n:\ \wb\leq \vb\}.$$
For a nonempty subset $A\subseteq [n]$, we have $\rho(A)=\sum_{i\in A}v_i$. Hence $\rho(A)<\rho(B)$ for any set $B\supsetneq A$. Thus $A$ is $\rho$-closed. If $|A|>1$, then $\rho(A)=\rho(A\setminus\{i\})+\rho(\{i\})$, which implies that $A$ is $\rho$-separable. Therefore the 
$\rho$-closed and $\rho$-inseparable subsets of $[n]$ are $\{1\},\{2\},\ldots,\{n\}$.
Moreover $\rho(\{i\})=v_i$ for $1\leq i\leq n$. 

Then by Theorem~\ref{basic},  $\Cl(R_\P)=\ZZ^{n-1} \oplus \ZZ/d\ZZ$, where $d=\gcd(v_1,\ldots,v_n)$. Moreover, by Corollary~\ref{gorenstein},  $R_\P$ is Gorenstein if and only if $1\leq v_1=v_2=\cdots=v_n\leq 2$.  
} 
\end{Example}

	The following corollary gives a necessary condition on a matroid $M$ for  $R_M$ to be  Gorenstein.
	Recall that a graph $G$ is called unmixed if all its maximal independent sets have the same cardinality. 
	
	\begin{Corollary}
		Let $M$ be a matroid on $[n]$ and let $G_M$ be the graph on $[n]$ whose edge set is the $1$-skeleton of $M$. If $R_M$ is Gorenstein, then $G_M$ is unmixed.  
	\end{Corollary}
	
	\begin{proof}
		Let $\rho\: 2 ^{[n]}\to \ZZ_+$ be the ground set rank function of $M$, where 
		\[
		\rho(A)=\max\{|A\cap F|\: F\in M\}.
		\] The maximal independent sets of $G_M$ are the parallel classes of $M$.
		 Hence each maximal independent set $A$ of $G_M$ is $\rho$-closed and $\rho$-inseparable with $\rho(A)=1$.

		% Let $A$ be a maximal independent set of $G_M$. Then %$|A\cap F|\leq 1$ for any $F\in M$. Hence 
		% $\rho(A)=1$. %For any set $B\subseteq [n]$ with $A\subsetneq B$, there exists $F\in M$ with $|F|=2$ such that $F\subseteq B$. Thus $\rho(B)\geq 2>\rho(A)$. This shows that $A$ is $\rho$-closed. 
		
	%	Now, let $A = X \cup Y$ with $X \cap Y = \emptyset$, where $X \neq \emptyset$ and $Y \neq \emptyset$. Since $\rho(X)\geq 1$ and $\rho(Y)\geq 1$, we have $\rho(A)<\rho(X)+\rho(Y)$. Therefore $A$ is $\rho$-inseparable. 
	By Corollary~\ref{gorenstein} if $R_M$ is Gorenstein, then there exists an integer $a$ such that for any  maximal independent set $A$ of $G_M$,  $|A|+1=a\rho(A)=a$. This means that $G_M$ is unmixed. 
	\end{proof}

\section{Classes of transversal polymatroids}
Let $\Ac = (A_1, \ldots,A_d)$ be a family of nonempty subsets of $[n]$ and suppose that $[n] = A_1\cup \ldots \cup A_d$.  It is {\em not} required that $A_i \neq A_j$ if $i \neq j$.  One defines the integer valued nondecreasing function $\rho_\Ac : 2^{[n]} \to \ZZ_+$ by setting
\[
\rho_\Ac(X) = |\{ k \in [n] : X \cap A_k \neq \emptyset \}|, \, \, \, \, \, X \subseteq [n]. 
\]
It follows from \cite[Exercise 12.2]{HHBook} that $\rho_\Ac$ is submodular and the set of bases of the discrete polymatroid $\Pc_\Ac \subset \ZZ_+^n$ arising from $\rho_\Ac$ is
\[
B_\Ac = \{e_{i_1} + \cdots + e_{i_d} : i_k \in A_k, 1 \leq k \leq d\} \subseteq \ZZ_+^n.
\]
One says that the discrete polymatroid $\Pc_\Ac$ is the {\em transversal polymatroid} presented by $\Ac$. 

\begin{Theorem}
\label{torsionfree}
Let $\Pc_\Ac$ be the transversal polymatroid presented by $\Ac = (A_1, \ldots, A_d)$ and suppose that there is an $i$ with $A_i \not\subseteq \cup_{i \neq j} A_j$.  Then ${\Cl}(R_{\Pc_\Ac})$ is free.
\end{Theorem}

\begin{proof}
Let $X = A_i \setminus \cup_{i \neq j} A_j$.  Then $X \neq \emptyset$ is $\rho_\Ac$-closed and $\rho_\Ac$-inseparable.  Since $\rho_\Ac(X) = 1$, it follows that ${\Cl}(R_{\Pc_\Ac})$ is torsion free.
\end{proof}

Now, fix $1 < i < n$ and write ${[n] \choose i}$ for the set of all $i$-element subsets of $[n]$.  Let $\Pc_{n,i}$ denote the transversal polymatroid presented by ${[n] \choose i}$, and let $\rho_{n,i}$ denote the rank function of $\Pc_{n,i}$.

\begin{Theorem}
	\label{transversal_theorem}
	$
	{\Cl}(R_{\Pc_{n,i}}) = \ZZ^{r-1},
	$
	where
	\[
	r = {n \choose 1} + {n\choose 2} + \cdots + {n \choose n-i}+1.
	\]
	
\end{Theorem}
\medskip 

In order to prove Theorem~\ref{transversal_theorem}, we need the following two lemmata. The first one is very easy to check. So the proof is omitted.

\begin{Lemma}
\label{rank_n_choose_i}
Let $X \subseteq [n]$.  Then
\[
\rho_{n,i}(X) = \left\{
\begin{array}{ll}
{n \choose i} - {n-|X| \choose i} & (|X| \leq n - i);\\
{n \choose i} & (|X| > n - i).
\end{array}
\right.
\]
\end{Lemma}

%\begin{proof}
%Let $A \in {[n] \choose i}$.  It is clear that $X \cap A \neq \emptyset$ if $|X| > n - i$.  Let $|X| = a\leq n - i$.  One can assume that $X = \{1, \ldots, a\}$.  It follows that $X \cap A = \emptyset$ if and only if $A \subset \{a+1, \ldots, n\}$.  Hence $\rho_{n,i}(X) = {n \choose i} - {n-a \choose i}$, as desired.    
%\end{proof} 

\begin{Lemma}
\label{closed_inseparable}
A subset $X \subseteq [n]$ is $\rho$-closed and $\rho$-inseparable if and only if $1 \leq |X| \leq n - i$ or $X=[n]$.
\end{Lemma}

\begin{proof}
Lemma \ref{rank_n_choose_i} says that $X \subseteq [n]$ is $\rho$-closed if and only if $1 \leq |X| \leq n - i$ or $X=[n]$.  It is clear that if $|X| = 1$, then $X$ is $\rho$-inseparable.  Let $1 < |X| \leq n - i$ or $X=[n]$ and $X = Y \cup Z$, where $Y \neq \emptyset, Z \neq \emptyset$ and $Y \cap Z = \emptyset$.  Let say, $1 \in Y, 2 \in Z$.  Then $\{1,2,\ldots, i\}$ intersects both $Y$ and $Z$.  Hence   
$
\rho_{n,i}(Y) + \rho_{n,i}(Z) > \rho_{n,i}(X).
$
Thus $X$ is $\rho$-inseparable.
\end{proof}

{\em Proof of Theorem \ref{transversal_theorem}}.
The result follows from Theorem
\ref{basic} and Lemma \ref{closed_inseparable} once we note that the greatest common divisor of the numbers 
\[
{n \choose i} - {n-1 \choose i}, {n \choose i} - {n-2 \choose i}, \ldots, {n \choose i} - {n - (n-i) \choose i}, {n\choose i}
\]
is $1$.

\begin{Example}
	\label{EXa}
	{\em
		Let $n = 7, i=4$.  Then 
		\[
		{7 \choose 4} - {6 \choose 4} = 20, \,
		{7 \choose 4} - {5 \choose 4} = 30, \, 
		{7 \choose 4} - {4 \choose 4} = 34, \,
		{7 \choose 4} = 35. \,  
		\]
		Hence
		\[
		{\Cl}(R_{\Pc_{7,4}}) = \ZZ^{r-1},
		\]
		where $r = {7 \choose 1} + {7 \choose 2} + {7 \choose 3}+1 = 64$.
		
		%\smallskip
		
		%(b) Let $n = 7, i=3$.  Then 
		%\[
		%{7 \choose 3} - {6 \choose 3} = 15, \,
		%{7 \choose 3} - {5 \choose 3} = 25, \, 
		%{7 \choose 3} - {4 \choose 3} = 31, \,
		%{7 \choose 3} - {4 \choose 3} = 34. %\,  
		%\]
		%Hence
		%\[
		%{\Cl}(R_{\Pc_{7,4}}) = \ZZ^{r-1},
		%\]
		%where $r = {7 \choose 1} + {7 \choose 2} + {7 \choose 3} + {7 \choose 4} = 98$.
	}
\end{Example}

%\begin{Proposition}
%\label{PROP}
%\noindent
%\begin{itemize}
%\item[(a)]
%Let $n \geq 4$.  Then %${\Cl}(R_{\Pc_{n,2}}) = \ZZ^{2^n - n - %3}$.
%\item[(b)]
%Let $n \geq 3$.  Then $
%{\Cl}(R_{\Pc_{n,n-1}}) = \ZZ^{n-1} %\oplus \ZZ/(n-1)\ZZ.
%$
%\item[(c)]
%Let $n \geq 3$.  Then $
%{\Cl}(R_{\Pc_{n,n-2}}) = \ZZ^{n^2/2} %\oplus \ZZ/d\ZZ$, where $d=(n-2)/2$ if %$n$ is even, and  $d=n-2$ if $n$ is %odd. 
%\end{itemize}
%\end{Proposition}

%\begin{proof}
%(a) One has
%\[
%r = {n \choose 1} + {n\choose 2} + %\cdots + {n \choose n-2} = 2^n - n - 2.
%\]
%Since
%\[
%{n \choose 2} - {n - 1 \choose 2} = n %- 1
%\]
%and
%\[
%{n \choose 2} - {n - 2 \choose 2} = 2 %n - 3= (n - 1) + (n - 2)
%\]
%are relatively prime, it follows that %${\Cl}(R_{\Pc_{n,2}})$ is torsion free.

%\smallskip

%(b) It follows that $X \subset [n]$ is %$\rho_{n,n-1}$-closed and %$\rho_{n,n-1}$-inseparable if and only %if $|X| = 1$.  Since %$\rho_{n,n-1}(\{i\}) = n - 1$, the %desired result follows.

%\smallskip

%(c) It follows that $X \subset [n]$ is $\rho_{n,n-2}$-closed and $\rho_{n,n-2}$-inseparable if and only if $|X| \leq 2$.  Thus $r = n + n(n-1)/2 = n^2/2$.  One has $$\rho_{n,n-2}(\{i\}) = {n-1 \choose n-3} = {n-1 \choose 2} = \frac{\,(n-1)(n-2)\,}{2}$$ and $$\rho_{n,n-2}(\{i, j\}) = {n \choose n-2} - 1 = {n \choose 2} - 1 = \frac{\,(n+1)(n-2)\,}{2}$$ if $i < j$.  Thus $d > 0$ is the greatest common divisor of $(n-1)(n-2)/2$ and $(n+1)(n-2)/2$. 
%\end{proof}

\begin{Theorem}
\label{EXb}
Let $\Ac = (\,\underbrace{A_1, \ldots, A_1}_{k_1}\,, \ldots, \, \underbrace{A_r, \ldots, A_r}_{k_r}\,)$, where $A_1\subsetneq \cdots\subsetneq A_r=[n]$, and let $\Pc_\Ac$ be the transversal polymatroid presented by $\Ac$. Then $\Cl(R_{\Pc_\Ac}) = \ZZ^{r-1} \oplus \ZZ/d\ZZ$, where $d=\gcd(k_1,\ldots,k_r)$. 
\end{Theorem} 

\begin{proof}
 Let $X\subseteq [n]$ be a $\rho_\Ac$-closed set and let $i$ be the smallest integer such that $X\cap A_{i}\neq \emptyset$. Since $X\subseteq [n]\setminus A_{i-1}$ and $\rho_\Ac([n]\setminus A_{i-1})=\rho_\Ac(X)$, we have  $X=[n]\setminus A_{i-1}$. Hence $[n], [n]\setminus A_1,\ldots,[n]\setminus A_{r-1}$ are the $\rho_\Ac$-closed sets.  Furthermore, each set $[n]\setminus A_{i}$ is $\rho_\Ac$-inseparable.  Indeed one has $\rho_\Ac([n]\setminus A_{i})=k_{i+1}+\cdots+k_r$. Let $[n]\setminus A_i=X\cup Y$ with $X\cap Y=\emptyset$, $X\neq \emptyset$ and $Y\neq \emptyset$.
			We may assume that $X\cap A_{i+1}\neq \emptyset$. Then $\rho_\Ac(X)=\rho_\Ac([n]\setminus A_i)$ and $\rho_\Ac(Y)\geq k_r$. Hence $\rho_\Ac(X)+\rho_\Ac(Y)>\rho_\Ac([n]\setminus A_i)$, as desired. 
			Hence $\Cl(R_{\Pc_\Ac}) = \ZZ^{r-1} \oplus \ZZ/d\ZZ$, where $d=\gcd(k_1,\ldots,k_r)$. 
\end{proof}

By choosing $k_1,\ldots,k_r=d$ in Theorem~\ref{EXb} we obtain the  following fundamental result in the present section.

\begin{Corollary}
\label{n_and_d}
Given $r > 0$ and $d > 0$, there exists a transversal polymatroid $\Pc_\Ac$ for which 
\[
{\Cl}(R_{\Pc_{\Ac}}) = \ZZ^{r-1} \oplus \ZZ/d\ZZ.
\]
\end{Corollary}

Let $G$ be a graph on $[n]$. A subset $C\subseteq [n]$ is called a {\em vertex cover} of $G$ if it intersects any edge of $G$.

	\begin{Example}
		\label{EXc}
{\em Let $n \geq 3$, and let $G$ be a finite simple connected graph on $[n]$ with the edge set $\{e_1, \ldots, e_s\}$. Suppose that $G$ is not the star graph $K_{1,n-1}$ (which is equivalent to $\cap_{i=1}^s e_i=\emptyset$). %Suppose that each vertex belongs to at least two edges and that each vertex does not belong to at least one edge.  
			Let $A_i = [n] \setminus e_i$ and $\Ac = (A_1, \ldots, A_s)$.  Let $\Pc_\Ac$ be the transversal polymatroid presented by $\Ac$. For any set $X\subseteq [n]$ we have $\rho_\Ac(X)>0$. Moreover, if $|X| > 2$, then $\rho_\Ac(X) = s$. %If $i$ is an isolated vertex of $G$, then $\rho_\Ac(\{i\})=s$ and 
			If $i$ is a leaf of $G$, then $\rho_\Ac(\{i\})=s-1$. Otherwise $\rho_\Ac(\{i\})\leq s-2$.  Let $i < j$.  If $\{i,j\}$ is an edge of $G$, then $\rho_\Ac(\{i,j\}) = s - 1$.  If $\{i,j\}$ is not an edge of $G$, then $\rho_\Ac(\{i,j\}) = s$.  Hence a set $\{i\}$ is $\rho_\Ac$-closed and $\rho_\Ac$-inseparable if and only if $\deg_G(i)\geq 2$.  Furthermore, $\{i,j\}$ is $\rho_\Ac$-closed if and only if $\{i,j\}$ is an edge of $G$.  Let $\{i,j\}$ be an edge of $G$ and suppose that $\{i,j\}$ is not $\rho_\Ac$-inseparable.  Then $\rho_\Ac(\{i,j\}) = \rho_\Ac(\{i\}) + \rho_\Ac(\{j\})$, which says that each edge contains either $i$ or $j$. In other words the set $\{i,j\}$ is a vertex cover of $G$.  Hence a set $\{i, j\}$ is $\rho_\Ac$-closed and $\rho_\Ac$-inseparable if and only if  $\{i,j\}$ is and edge of $G$ and it is not a vertex cover of $G$.  Clearly, $[n]$ is $\rho_\Ac$-closed. Let $[n]=X\cup Y$ with $X\neq \emptyset$, $Y\neq \emptyset$ and $X\cap Y=\emptyset$. If $n\geq 5$, then either $|X|>2$ or $|Y|>2$. Hence $\rho_\Ac(X)=s$ or $\rho_\Ac(Y)=s$. Since $\rho_\Ac(X)>0$ and $\rho_\Ac(Y)>0$, we get $\rho_\Ac(X)+\rho_\Ac(Y)>s=\rho_\Ac([n])$. Hence $[n]$ is $\rho_\Ac$-inseparable.
			If $n=3$, since $G\neq K_{1,2}$ and $G$ is connected we have $G=K_3$. Then for $X=\{1,2\}$ and $Y=\{3\}$ we have $\rho_\Ac([3])=\rho_\Ac(X)+\rho_\Ac(Y)$. Hence $[n]$ is $\rho_\Ac$-separable. Now, let $n=4$ and consider a nontrivial partition $[n]=X\cup Y$. If $|X|>2$ or  $|Y|>2$, then $\rho_\Ac(X)+\rho_\Ac(Y)>\rho_\Ac([n])$. Let $|X|=|Y|=2$. Then $\rho_\Ac(X)\geq s-1$ and $\rho_\Ac(Y)\geq s-1$ and then $\rho_\Ac(X)+\rho_\Ac(Y)\geq 2s-2$. Since $G$ is connected, $s\geq n-1=3$. Hence $2s-2> s=\rho_\Ac([n])$. Then we conclude that $[n]$ is $\rho_\Ac$-inseparable if and only if $n>3$.
			Hence for $n>3$, $\Cl(R_{\Pc_\Ac})$ is free, since $s - 1$ and $s$ are relatively prime and there exists at least one edge which is not a vertex cover of $G$. Moreover, the rank of $\Cl(R_{\Pc_\Ac})$ is $n-l+m$, where $l$ is the number of leaves of $G$ and $m$ is the number of edges of $G$ which are not vertex covers of $G$. Finally, if $n=3$, then $\Cl(R_{\Pc_\Ac})$ is free of rank $2$.  
}
\end{Example}

The following theorem characterizes transversal polymatroids whose divisor class groups are finite.	
	\begin{Theorem}
	\label{finite}
	Let $\Pc_\Ac$ be the transversal polymatroid presented by $\Ac = (A_1, \ldots, A_s)$.  Then $\Cl(R_{\Pc_\Ac}) = \ZZ/d\ZZ$ with $d > 0$ if and only if  $A_1 = \cdots = A_s = [n]$ and $s=d$.  
\end{Theorem}	

\begin{proof}
The result follows from Theorem \ref{basic} and Lemma \ref{closed_inseparable2}.
\end{proof}

\begin{Lemma}
\label{closed_inseparable2}
Let $\Pc_\Ac$ be the transversal polymatroid presented by $\Ac = (A_1, \ldots, A_s)$.  Then there is a unique $\rho_\Ac$-closed and $\rho_\Ac$-inseparable subset of $[n]$ if and only if $A_1 = \cdots = A_s = [n]$.  
\end{Lemma}

\begin{proof}
If $A_1 = \cdots = A_s = [n]$, then $[n]$ is the unique $\rho_\Ac$-closed and $\rho_\Ac$-inseparable subset of $[n]$.  

Now, suppose that there is a unique $\rho_\Ac$-closed and $\rho_\Ac$-inseparable subset of $[n]$. 
Since $[n] =\cup_{i=1}^d A_i$, we have 
$\rho_\Ac(\{j\})>0$ for any $j\in [n]$. 
Therefore for each element $j$, the set $B_j=\{k:\rho_\Ac(\{j,k\})=\rho_\Ac(\{j\})\}$ is  $\rho_\Ac$-closed and $\rho_\Ac$-inseparable. Indeed, for any nonempty subset $B\subseteq B_j$ with $j\in B$, one has $\rho_\Ac(B)=\rho_\Ac(\{j\})$. This shows that $B_j$ is $\rho_\Ac$-inseparable.  For any set $C$ with $B_j\subsetneq C$, we have $\rho_\Ac(C)>\rho_\Ac(\{j\})=\rho_\Ac(B_j)$. This implies that $B_j$ is $\rho_\Ac$-closed.
 Since there is only one $\rho_\Ac$-closed and $\rho_\Ac$-inseparable set, it follows that $\rho_\Ac(\{j\})=\rho_\Ac(\{j,k\})=\rho_\Ac(\{k\})$ for all $j$ and $k$. So all elements belong to $A_i$ for each $i$, and hence each $A_i$ is $[n]$. 

%Then $[n]$ must be $\rho_\Ac$-inseparable.  
%In fact, if $[n]$ is not $\rho_\Ac$-inseparable, then $[n] = X \cup Y$ with $X \cap Y = \emptyset$, where $X \neq \emptyset, Y \neq \emptyset$, and with $\rho_\Ac([n]) = \rho_\Ac(X) + \rho_\Ac(Y)$.  Hence each $A_i$ is contained in either $X$ or $Y$.  Thus in particular each of $X$ and $Y$ is $\rho_\Ac$-closed.  By continuing these procedure, one can find a $\rho_\Ac$-closed and $\rho_\Ac$-inseparable subset of $X$ and $\rho_\Ac$-closed and $\rho_\Ac$-inseparable subset of $Y$.  This contradicts the uniqueness of $\rho_\Ac$-closed and $\rho_\Ac$-inseparable subset.  Hence $[n]$ is $\rho_\Ac$-inseparable.  

%Now, suppose there is $i$ with $A_i \neq [n]$.  Let $j \not\in A_i$ and $Z$ a maximal subset of $[n]$ with $j \in Z$ and $\rho_\Ac(\{j\}) = \rho_\Ac(Z)$.  Then $\rho_\Ac(Z) <s$ and $Z$ is $\rho_\Ac$-closed and $\rho_\Ac$-inseparable.  Hence both $[n]$ and $Z$ are  $\rho_\Ac$-closed and $\rho_\Ac$-inseparable, a contradiction.  Hence $A_1 = \cdots = A_s = [n]$, as desired. 
\end{proof}

\begin{Theorem}
	\label{2_prime_THEOREM}
	Let $\Pc_\Ac$ be the transversal polymatroid presented by $\Ac = (A_1, \ldots, A_s)$.  Then $\Cl(R_{\Pc_\Ac})$ is $\ZZ$ or $\ZZ \oplus \ZZ/d \ZZ$ if and only if one of the following conditions is satisfied:
	\begin{itemize}
		\item[(i)]
		There is a decomposition $[n] = B \cup C$, where $B \neq \emptyset, C \neq \emptyset, B \cap C = \emptyset$ for which, for some $1 \leq q < s$, one has $A_1 = \cdots = A_q = B$ and $A_{q+1} = \cdots = A_s = C$.
		\item[(ii)]
		There is a nonempty subset $A \subsetneq [n]$ for which, for some $1 \leq q < s$, one has $A_1 = \cdots = A_q = A$ and $A_{q+1} = \cdots = A_s = [n]$. 
	\end{itemize}  
	Moreover,   $\Cl(R_{\Pc_\Ac}) = \ZZ$ if and only if $q$ and $s-q$ are relatively prime.  If the greatest common divisor of $q$ and $s-q$ is $d > 1$, then $\Cl(R_{\Pc_\Ac}) = \ZZ \oplus \ZZ/d \ZZ$.
\end{Theorem}

\begin{proof}
The statement follows from  from Theorem \ref{basic} and the following lemma. 
\end{proof}

\begin{Lemma}
	\label{2_prime}
	Let $\Pc_\Ac$ be the transversal polymatroid presented by $\Ac = (A_1, \ldots, A_s)$, where $(A_1, \ldots, A_s) \neq ([n], \ldots, [n])$.  Then there are exactly two $\rho_\Ac$-closed and $\rho_\Ac$-inseparable subsets of $\Pc_\Ac$ if and only if one of the following conditions is satisfied:
	\begin{itemize}
		\item[(i)]
		There is a decomposition $[n] = B \cup C$, where $B \neq \emptyset, C \neq \emptyset, B \cap C = \emptyset$, for which each $A_i$ is equal to either $B$ or $C$. 
		\item[(ii)]
		There is a nonempty subset $\Omega \subsetneq [n]$, for which each $A_i$ is equal to either $[n] \setminus \Omega$ or $[n]$.
	\end{itemize}
\end{Lemma}

\begin{proof}
	Suppose there are exactly two $\rho_\Ac$-closed and $\rho_\Ac$-inseparable subsets of $\Pc_\Ac$.
	
	First, suppose that $[n]$ is not $\rho_\Ac$-inseparable.  Then there is a decomposition $[n] = B \cup C$, where $B \neq \emptyset, C \neq \emptyset, B \cap C = \emptyset$, for which each $A_i$ is contained in either $B$ or $C$.  Each of $B$ and $C$ is $\rho_\Ac$-closed.  We claim that each of $B$ and $C$ is $\rho_\Ac$-inseparable.  If, say, $B$ is not $\rho_\Ac$-inseparable, then there is a decomposition $B = D \cup E$, where $D \neq \emptyset, E \neq \emptyset, D \cap E = \emptyset$, for which each $A_i \subset B$ is contained in either $D$ or $E$.  Let $k \in D$ and $Y$ the unique maximal subset of $[n]$ with $k \in Y$ and with $\rho_{\Ac}(\{k\}) = \rho_{\Ac}(Y)$.  Then $Y$ is $\rho_\Ac$-closed and $\rho_\Ac$-inseparable with $Y \subset D$.  This observation guarantees that each of $D, E$ and $C$ contains a $\rho_\Ac$-closed and $\rho_\Ac$-inseparable subset of $\Pc_\Ac$, a contradiction.  Hence both $B$ and $C$ are $\rho_\Ac$-closed and $\rho_\Ac$-inseparable.
	
	Now, let, say, $A_i \subsetneq B$ and $j \in B \setminus A_i$.  Let $X$ denote the unique maximal subset of $B$ with $j \in X$ and with $\rho_{\Ac}(\{j\}) = \rho_{\Ac}(X)$.  Then $X$ is $\rho_\Ac$-closed and $\rho_\Ac$-inseparable with $X \cap A_i = \emptyset$.  Hence $B, C$ and $X$ are $\rho_\Ac$-closed and $\rho_\Ac$-inseparable, a contradiction. 
	
	Second, suppose that $[n]$ is $\rho_\Ac$-closed and $\rho_\Ac$-inseparable.  
As was shown in the proof of Lemma~\ref{closed_inseparable2} for each $j$, the set $B_j=\{k:\rho_\Ac(\{j,k\})=\rho_\Ac(\{j\})\}$ is  $\rho_\Ac$-closed and $\rho_\Ac$-inseparable.	If $B_j=[n]$ for all $j$, then $A_1=\cdots=A_s=[n]$ which is not the case by our assumption. Hence there exists $j$ with $B_j\neq [n]$. Set $\Omega=B_j$. Then for any $k$ with $B_k\neq [n]$, we have $\Omega=B_k$. Hence  $\Omega=\{j\in [n]:\ B_j\neq [n]\}$ and $j\notin \Omega$ if and only if $j\in\cap_{i=1}^s A_i$. Hence $[n]\setminus \Omega\subseteq A_i$ for all $i$. Moreover, for each $i$ we have either $\Omega\subset A_i$ or $\Omega\cap A_i=\emptyset$. It follows that for each $i$ either $A_i = [n]$ or $A_i = [n] \setminus \Omega$.  Since $[n] = A_1 \cup \cdots \cup A_s$, there is an $i$ with $A_i = [n]$.   
	
%Let $\Omega$ denote the set of integers $j \in [n]$ with $j \not\in A_i$ for some $i$.
%Then $\Omega \neq \emptyset$.  For each $j \in \Omega$, one writes $X_j$ for a unique maximal subset of $[n]$ with $j \in X_j$ and with $\rho_{\Pc_\Ac}(\{j\}) = \rho_{\Pc_\Ac}(X_j)$.  Then $X_j \subsetneq [n]$ and $X_j$ is $\rho_\Ac$-closed and $\rho_\Ac$-inseparable.  It follows from the assumption that $X_j = X_{j'}$ for $j, j' \in \Omega$.  In particular, each $A_i$ satisfies either $\Omega \cap A_i = \emptyset$ or $A_i \subset \Omega$.  It then follows that for each $i$ one has either $A_i = [n]$ or $A_i = [n] \setminus \Omega$.  Since $[n] = A_1 \cup \cdots \cup A_s$, there is an $i$ with $A_i = [n]$.  
	
	\medskip
	
	On the other hand, if (i) or (ii) is satisfied, then clearly there are exactly two $\rho_\Ac$-closed and $\rho_\Ac$-inseparable subsets of $\Pc_\Ac$.       
\end{proof}

\section{Gorenstein polymatroids of Veronese type}
% E.~De Negri and T.~Hibi, Gorenstein algebras of Veronese type, {\em J. Algebra} {\bf 193} (1997), 629--639.
\def\sb{{\mathbf s}}
Fix an integer $d$ and a sequence $\sb = (s_1, \ldots s_n)$ of integers with $1 \leq s_1 \leq \cdots \leq s_n \leq d$ and $d < \sum_{i=1}^{n} s_i$.  The discrete polymatroid {\em of Veronese type} $(\sb, d)$ is the discrete polymatroid \[
\Pc_{\sb,d} = \{\vb \in \ZZ^n_{+} : v_i \leq s_i, |\vb| \leq d\}.
\]
%Discrete polymatroids of Veronese type may be viewed as skeleta of transversal polymatroids.
Let $\rho_{\sb,d} = \rho_{\Pc_{\sb,d}}$ denote the rank function of $\Pc_{\sb,d}$.

Corollary \ref{gorenstein} enables us to classify Gorenstein discrete polymatroids of Veronese type.

\begin{Theorem}
	\label{Gorenstein_Veronese}
	The toric ring $R_{\Pc_{\sb, d}}$ is Gorenstein if and only if one of the following conditions is satisfied:
	\begin{itemize}
		\item[(i)]
		each $s_i = 2$ and $n = d - 1 \geq 2$;
		\item[(ii)]
		each $s_i = 1$ and $n = 2d - 1 \geq 3$.
	\end{itemize}
\end{Theorem}
In order to prove Theorem~\ref{Gorenstein_Veronese} we need the following 
\begin{Lemma}
\label{veronese_rank}
The $\rho_{\sb,d}$-closed and $\rho_{\sb,d}$-inseparable subsets are each $\{i\}$ and $[n]$.  Furthermore, $\rho_{\sb,d}(\{i\}) = s_i$ and $\rho_{\sb,d}([n]) = d$. 
\end{Lemma}

\begin{proof}
If $\emptyset \neq A \subset [n]$, then $\rho_{\sb,d}(A) = \min\{\sum_{i \in A} s_i, d\}$.  Let $A \subset [n]$ with $1 <|A| < n$.  If $\rho_{\sb,d}(A) = \sum_{i \in A}s_i$, then $A$ cannot be $\rho_{\sb,d}$-inseparable.  If $\rho_{\sb,d}(A) = d$, then $A$ cannot be $\rho_{\sb,d}$-closed.

First, one shows that $[n]$ is $\rho_{\sb,d}$-closed and $\rho_{\sb,d}$-inseparable.  Clearly $[n]$ is $\rho_{\sb,d}$-closed.  Let $[n] = A \cup B$ with $A \neq \emptyset, B \neq \emptyset, A \cap B = \emptyset$.  Then $\rho_{\sb,d}(A) = \min\{\sum_{i \in A} s_i, d\}$ and $\rho_{\sb,d}(B) = \min\{\sum_{i \in B} s_i, d\}$.  Hence $\rho_{\sb,d}([n]) < \rho_{\sb,d}(A) +\rho_{\sb,d}(B)$.  Thus $[n]$ is $\rho_{\sb,d}$-inseparable. 

Second, one shows that each $\{i\}$ is $\rho_{\sb,d}$-closed and $\rho_{\sb,d}$-inseparable.  Clearly $\{i\}$ is $\rho_{\sb,d}$-inseparable.  Let $j \neq i$.  Then $\rho_{\sb,d}(\{i,j\}) = \min\{s_i + s_j, d\}$.  Since $s_i < d$ and $s_j > 0$, one has $\rho_{\sb,d}(\{i\}) < \rho_{\sb,d}(\{i,j\})$.  Hence $\{i\}$ is $\rho_{\sb,d}$-closed.
\end{proof}

{\em Proof of Theorem~\ref{Gorenstein_Veronese}}.  
It follows that $R_{\Pc_{\sb, d}}$ is Gorenstein if and only if there is an integer $a \geq 1$ for which
\[
\rho_{\sb,d}(\{i\}) = s_i = \frac{2}{\,a\,}, \, \, \, \rho_{\sb,d}([n]) = d = \frac{n+1}{\,a\,}.
\]
Let $a = 1$.  Then each $s_i = 2$ and $n = d-1 \geq 2$.  Let $a=2$.  Then each $s_i = 1$ and $n = 2d - 1 \geq 3$, as desired.

\end{document}